\documentclass[review]{elsarticle}

\usepackage{lineno,hyperref}
\usepackage{amsmath}
\usepackage{amsthm}
\usepackage{amssymb}
\usepackage[T1]{fontenc}
\usepackage[utf8]{inputenc}
\modulolinenumbers[5]

\journal{Indagationes Mathematicae}

\renewcommand{\div}{{\textrm{div}\:}}
\newcommand{\Q}{\mathbb{Q}}
\newcommand{\QQ}{\mathbb{Q}}

\newcommand{\Qbar}{{\overline{\Q}}}
\newcommand{\DD}{\mathcal{D}_m}
\newcommand{\Dbar}{\overline{\mathcal{D}}_m}
\newcommand{\Gal}{\textrm{Gal}(\Qbar/\Q)}

\theoremstyle{plain}
\newtheorem{theorem}{Theorem}
\newtheorem{proposition}[theorem]{Proposition}
\newtheorem{lemma}[theorem]{Lemma}

\theoremstyle{remark}
\newtheorem{remark}[theorem]{Remark}

\begin{document}

\begin{frontmatter}

\title{Rational $D(q)$-quadruples}

\author{Goran Dra\v zi\' c\fnref{fnote1}}
\ead{gdrazic@pbf.hr} 
\author{Matija Kazalicki\fnref{fnote2}}
\ead{matija.kazalicki@math.hr}
\fntext[fnote1]{Faculty of Food Technology and Biotechnology, University of Zagreb, Pierottijeva 6, 10000 Zagreb, Croatia.}
\fntext[fnote2]{Department of Mathematics, University of Zagreb, Bijeni\v cka cesta 30, 10000 Zagreb,
Croatia.}

\begin{abstract}
For a rational number $q$, a  \emph{rational $D(q)$-$n$-tuple} is a set of $n$ distinct nonzero rationals $\{a_1, a_2, \dots, a_n\}$ such that $a_ia_j+q$ is a square for all $1 \leqslant i < j \leqslant n$. For every $q$ we find all rational $m$ such that there exists a $D(q)$-quadruple with product $a_1a_2a_3a_4=m$. We describe all such quadruples using points on a specific elliptic curve depending on $(q,m).$
\end{abstract}


\end{frontmatter}

\linenumbers

\section{Introduction}

Let $q\in \QQ$ be a nonzero rational number. A set of $n$ distinct nonzero rationals $\lbrace a_1, a_2, \dots, a_n\rbrace$ is called a rational $D(q)$-$n$-tuple if $a_ia_j+q$ is a square for all $1 \leqslant i<j\leqslant n.$ If $\lbrace a_1,a_2,\dots,a_n\rbrace$ is a rational $D(q)$-$n$-tuple, then for all $r\in \QQ,\lbrace ra_1,ra_2,\dots,ra_n\rbrace$ is a $D(qr^2)$-$n$-tuple, since $(ra_1)(ra_2)+qr^2=(a_1a_2+q)r^2$. With this in mind, we restrict to square-free integers $q.$ If we set $q=1$ then such sets are called rational Diophantine $n$-tuples.

The first example of a rational Diophantine quadruple was the set $$\left\lbrace\frac{1}{16}, \frac{33}{16}, \frac{17}{4}, \frac{105}{16}\right\rbrace$$ found by Diophantus, while the first example of an integer Diophantine quadruple, the set 
\[
\lbrace 1,3,8,120 \rbrace
\] is due to Fermat. 

In the case of integer Diophantine $n$-tuples, it is known that there are infinitely many Diophantine quadruples (e.g. $\{k-1, k+1, 4k, 16k^3-4k\},$ for  $k\geq 2$). Dujella \cite{dujella2004there} showed there are no Diophantine sextuples and only finitely many Diophantine quintuples, while recently  He, Togb\' e and Ziegler \cite{he2019there} proved there are no integer Diophantine quintuples, which was a long standing conjecture.

Gibbs \cite{gibbs2006some} found the first example of a rational Diophantine sextuple using a computer, and Dujella, Kazalicki, Miki\' c and Szikszai \cite{dujella2017there} constructed infinite families of rational Diophantine sextuples. Dujella and Kazalicki parametrized Diophantine quadruples with a fixed product of elements using triples of points on a specific elliptic curve, and used that parametrization for counting Diophantine quadruples over finite fields \cite{dujella2016diophantine} and for constructing rational sextuples \cite{dujella2017more}. There is no known rational Diophantine septuple. 

Regarding rational $D(q)$-$n$-tuples, Dujella \cite{dujella2000note} has shown that there are infinitely many rational $D(q)$-quadruples for any $q\in \QQ.$ Dujella and Fuchs in \cite{dujella2012problem} have shown that, assuming the Parity Conjecture, for infinitely squarefree integers $q\neq 1$ there exist infinitely many rational $D(q)$-quintuples. There is no known rational $D(q)$-sextuple for $q\neq a^2, a\in \QQ.$

Our work uses a similar approach Dujella and Kazalicki had in \cite{dujella2016diophantine} and \cite{dujella2017more}.

Let $\{a,b,c,d\}$ be a rational $D(q)$-quadruple, for a fixed nonzero rational $q,$ such that $$ ab+q=t_{12}^2,\quad ac+q=t_{13}^2,\quad ad+q=t_{14}^2,$$ $$bc+q=t_{23}^2,\quad bd+q=t_{24}^2,\quad cd+q=t_{34}^2.$$

Then $(t_{12},t_{13},t_{14},t_{23},t_{24},t_{34},m=abcd)\in \QQ^7$ defines a rational point on the algebraic variety $\mathcal{C}$ defined by the equations $$(t_{12}^2-q)(t_{34}^2-q)=m,$$  $$(t_{13}^2-q)(t_{24}^2-q)=m,$$ $$(t_{14}^2-q)(t_{23}^2-q)=m.$$

The rational points $(\pm t_{12}, \pm t_{13}, \pm t_{14}, \pm t_{23}, \pm t_{24}, \pm t_{34}, m)$ on $\mathcal{C}$ determine two rational $D(q)$ quadruples $\pm (a,b,c,d)$ $\Big($specifically,  $a^2=\frac{(t_{12}^2-q)(t_{13}^2-q)}{t_{23}^2-q}\Big)$ if $a,b,c,d$ are rational, distinct and nonzero. 

Any point $(t_{12},t_{13},t_{14},t_{23},t_{24},t_{34},m) \in\mathcal{C}$ corresponds to three points $Q'_1=(t_{12},t_{34}), Q'_2=(t_{13},t_{24})$ and  $Q'_3=(t_{14},t_{23})$ on the curve\[
\mathcal{D}_m\colon (X^2-q)(Y^2-q)=m.
\]
If $\DD(\QQ)=\emptyset,$ there are no rational $D(q)$-quadruples with product of elements equal to $m,$ so we assume there exists a point $P_1=(x_1,y_1) \in \DD(\QQ).$

The curve $\DD$ is a curve of genus $1$ unless $m=0$ or $m=q^2,$ which we assume from now on. Since we also assumed a point $P_1 \in \DD(\QQ),$ the curve $\DD$ is birationally equivalent to the elliptic curve 
\[
E_m\colon W^2=T^3+(4q^2-2m)T^2+m^2T
\] via a rational map $f\colon \mathcal{D}_m \to E_m$ given by 
\begin{align*}
T&=(y_1^2-q)\cdot\frac{2x_1(y^2-q)x+(x_1^2+q)y^2+x_1^2y_1^2-2x_1^2q-y_1^2q}{(y-y_1)^2}, \\
W&=T\cdot\frac{2y_1x(q-y^2)+2x_1y(q-y_1^2)}{y^2-y_1^2}.
\end{align*}
Note that $f$ maps $(x_1,y_1)$ to the point at infinity $\mathcal{O}\in E_m(\QQ)$, it maps $(-y_1,x_1)$ to a point of order four, $R=(m,2mq) \in E_m(\QQ)$, 
and maps $(-x_1,y_1)$ to 
\[
S=\left(\frac{y_1^2(x_1^2-q)^2}{x_1^2},\frac{qy_1(x_1^2+y_1^2)(x_1^2-q)^2}{x_1^3}\right)\in E_m(\QQ),
\] which is generically a point of infinite order.

We have the following associations
\[
(a,b,c,d)\dashleftarrow \rightarrow \text{a point on }\mathcal{C}(\Q)\longleftrightarrow (Q'_1,Q'_2,Q'_3) \in \mathcal{D}_m(\mathbb{Q})^3.
\]
In order to obtain a rational $D(q)$-quadruple from a triple of points on $\DD(\QQ),$ we must satisfy the previously mentioned conditions: $a,b,c,d$ must be rational, mutually disjoint and nonzero.

It is easy to see that if one of them is rational, then so are the other three (i.e. $b=\frac{t_{12}^2-q}{a}$), and that they will be nonzero when $m\neq 0,$ since $m=abcd.$

The elements of the quadruple $(a,b,c,d)$ corresponding to the triple of points $(Q'_1, Q'_2, Q'_3)$ are distinct, if no two of the points $Q'_1, Q'_2, Q'_3$ can be transformed from one to another via changing signs and/or switching coordinates. For example, the triple $(t_{12},t_{34}),(-t_{34},t_{12}),(t_{14},t_{23})$ would lead to $a=d.$ This condition on points in $\DD$ is easily understood on points in $E_m.$

Assume $P\in E_m \leftrightarrow (x,y)\in \mathcal{D}_m,$ that is, $f(x,y)=P.$ Then 
\begin{equation}\label{ness}
    S-P\leftrightarrow (-x,y),\quad P+R\leftrightarrow (-y,x).
\end{equation}
The maps $P\mapsto S-P$ and $P\mapsto P+R$ generate a group $G$ of translations on $E_m$, isomorphic to $\mathbb{Z}/2\mathbb{Z} \times \mathbb{Z}/4\mathbb{Z},$ and $G$ induces a group action on $E_m(\Qbar).$ In order to obtain a quadruple from the triple $(Q_1,Q_2,Q_3)\in E_m(\Q)^3$, such that the elements of the quadruple are distinct, the orbits $G\cdot Q_1, G\cdot Q_2, G\cdot Q_3$ must be disjoint. This is because the set of points in $\mathcal{D}_m$ corresponding to $G\cdot P$ is $\{(\pm x,\pm y),(\pm y,\pm x)\}.$ We say that such a triple of points satisfies the non-degeneracy criteria.

Let $\Dbar$ denote the projective closure of the curve $\DD$ defined by
\[
\Dbar \colon (X^2-qZ^2)(Y^2-qZ^2)=mZ^4.
\] The map $f^{-1} \colon E_m \to \Dbar$ is a rational map, and since the curve $E_m$ is smooth, the map is a morphism \cite[II.2.1]{silverman2009arithmetic}. The map $x\circ f^{-1}\colon E_m \to \mathbb{A}^1$ given by 
\[
x\circ f^{-1}(P)=\frac{X\circ f^{-1}(P)}{Z\circ f^{-1}(P)}
\] 
has a pole in points $P_0$ such that $f^{-1}(P_0)=[1:0:0],$ and is regular elsewhere. The map $y\circ f^{-1}\colon E_m \to \mathbb{A}^1$ given by 
\[
y\circ f^{-1}(P)=\frac{Y\circ f^{-1}(P)}{Z\circ f^{-1}(P)}
\] 
has a pole in points $P_2$ such that $f^{-1}(P_2)=[0:1:0],$ and is regular elsewhere.
We define the rational map $g\colon E_m \to \mathbb{A}^1$ by
\[
g(P)=(x_1^2-q)\cdot \left(\left(x\circ f^{-1}(P)\right)^2-q\right).
\]
The map $g$ has a pole in the same points as the map $x\circ f^{-1},$ and is regular elsewhere.

The maps $f$ and $g$ depend on a fixed point $P_1\in\mathcal{D}_m.$ We omit noting this dependency and simply denote these maps by $f$ and $g$. The motivation for the map $g$ is \cite[2.4, Proposition 4]{dujella2017more}. Dujella and Kazalicki use the $2$-descent homomorphism in the proof of Proposition 4, we will use $g$ for similar purposes.

\begin{theorem}\label{thm:1}
Let $(x_1,y_1)\in \DD(\QQ)$ be the point used to define the map $f\colon \DD \to E_m.$ If $(Q_1,Q_2,Q_3)\in E_m(\QQ)^3$ is a triple satisfying the non-degeneracy criteria such that $(y_1^2-q)\cdot g(Q_1+Q_2+Q_3)$ is a square, then the numbers
$$a=\pm\left(\frac{1}{m}\frac{g(Q_1)}{(x_1^2-q)}\frac{g(Q_2)}{(x_1^2-q)}\frac{g(Q_3)}{(x_1^2-q)}\right)^{1/2},$$ $$b=\frac{g(Q_1)}{a(x_1^2-q)}, c=\frac{g(Q_2)}{a(x_1^2-q)}, d=\frac{g(Q_3)}{a(x_1^2-q)}$$ are rational and form a \emph{rational $D(q)$-quadruple} such that $abcd=m$.

Conversely, assume $(a,b,c,d)$ is a \emph{rational $D(q)$-quadruple}, such that $m=abcd$. If the triple $(Q_1,Q_2,Q_3)\in E_m(\QQ)^3$ corresponds to $(a,b,c,d)$, then $(y_1^2-q)g(Q_1+Q_2+Q_3)$ is a square.
\end{theorem}
   
It is not true that the existence of a rational point on $\mathcal{D}_m(\QQ)$ implies the existence of a rational $D(q)$-quadruple with product $m.$ Examples with further clarification are given in Section \ref{sec:4}. The following classification theorem holds:
\begin{theorem}\label{thm:2}
There exists a rational $D(q)$-quadruple with product $m$ if and only if 
\[m=(t^2-q)\left(\frac{u^2-q}{2u}\right)^2
\] for some rational parameters $(t,u).$
\end{theorem}

In Section \ref{sec:2} we study properties of the function $g$ which we then use in Section \ref{sec:3} to prove Theorems \ref{thm:1} and \ref{thm:2}. In Section \ref{sec:4}, we give an algorithm on how to determine whether a specific $m$ such that $\mathcal{D}_m(\QQ)\neq \emptyset$ admits a rational $D(q)$-quadruple with product $m.$ We conclude the section with an example of an infinite family.

\section{Properties of the function $g$}\label{sec:2}

In this section, we investigate the properties of the function $g$ which we will use to prove the main theorems. The following proposition describes the divisor of $g.$

\begin{proposition}\label{prop:3} The divisor of $g$ is $$\div g=2(S_1)+2(S_2)-2(R_1)-2(R_2),$$ 
where
$S_1,R_1,S_2,R_2\in E_m(\QQ(\sqrt{q}))$ with coordinates
\begin{align*}
S_1&=(~(y_1^2-q)(x_1-\sqrt{q})^2, \hspace{5.8pt}\quad 2y_1\sqrt{q}~(y_1^2-q)(x_1-\sqrt{q})^2~),\\
R_1&=(~(x_1^2-q)(y_1+\sqrt{q})^2, \hspace{5.8pt}\quad 2x_1\sqrt{q}~(x_1^2-q)(y_1+\sqrt{q})^2~),\\
S_2&=(~(y_1^2-q)(x_1+\sqrt{q})^2,~ -2y_1\sqrt{q}~(y_1^2-q)(x_1+\sqrt{q})^2~),\\
R_2&=(~(x_1^2-q)(y_1+\sqrt{q})^2,~ -2x_1\sqrt{q}~(x_1^2-q)(y_1-\sqrt{q})^2~).
\end{align*}

The points $S_1, S_2, R_1$ and $R_2$ satisfy the following identities:
\begin{align*}
2S_1&=2S_2=f(x_1,-y_1)=S+2R,\\
2R_1&=2R_2=f(-x_1,y_1)=S,\\
S_1+R&=R_1,\quad R_1+R=S_2,\quad S_2+R=R_2,\quad R_2+R=S_1.\end{align*}
\end{proposition}

\begin{proof} We seek zeros and poles of $g.$ The poles of $g$ are the same as the poles of $x\circ f^{-1}.$ To find zeros of $g$, notice that 
\[
(x\circ f^{-1}(P))^2-q=\frac{m}{(y\circ f^{-1}(P))^2-q},
\] so all we need to find are poles of $y \circ f^{-1}.$
 
The zeros of $x\circ f^{-1}$ are points on $E_m$ which map to affine points on $\Dbar$ that have zero $x$-coordinate. We can easily calculate such points. If $x=0,$ then $y^2=\frac{q^2-m}{q}.$ Denote $K=\sqrt{\frac{q^2-m}{q}}.$ We know $K\neq 0,$ since $m\neq q^2.$

The zeros of $x\circ f^{-1}$ are the points $f(0,K), f(0,-K)\in E_m(\Qbar),$ which are different since $K\neq 0.$ Since $x\circ f^{-1}$ is of degree two, both zeros are of order one. We conclude $x\circ f^{-1}$ has either one double pole, or two poles of order one.

Similarly, the zeros of $y\circ f^{-1}$ are the points $f(K,0), f(-K,0)\in E_m(\Qbar),$ both of order one. The map $y\circ f^{-1}$ also has either a double pole or two poles of order one.

Assume the point $P_0 \in E_m$ maps to a non-affine point in $\Dbar.$ This means that $Z\circ f^{-1}(P_0)=0,$ and at least one of the projective coordinate functions $X\circ f^{-1}, Y\circ f^{-1}$ is nonzero at $P_0.$ It follows that $P_0$ is a pole of at least one of the maps $x\circ f^{-1}, y\circ f^{-1}.$

Let $P_0 \in E_m$ be a pole of one of the maps $x\circ f^{-1}, y\circ f^{-1}.$ None of the points $ f^{-1}(P_0),$  $f^{-1}(P_0+R),  f^{-1}(P_0+2R),  f^{-1}(P_0+3R)$ are affine points on $\Dbar$ because if one of them is an affine point, then they all are, since the map $P\mapsto P+R$ viewed on $\Dbar$ maps affine points to affine points. We conclude that each of the points $P_0, P_0+R, P_0+2R, P_0+3R$ is a pole of one of the maps $x\circ f^{-1}, y\circ f^{-1}$ and with the previous claims we have that $x\circ f^{-1}, y\circ f^{-1}$ both have two poles of order one.

The map $P\mapsto S-P,$ viewed on $\Dbar$, also maps affine points to affine points. Similarly as above, the points $f^{-1}(S-P_0), f^{-1}(S-P_0+R), f^{-1}(S-P_0+2R), f^{-1}(S-P_0+3R)$ are not affine in $\Dbar,$ because the point $f^{-1}(P_0)$ would be affine as well. The sets $\lbrace P_0, P_0+R, P_0+2R, P_0+3R \rbrace$ and $\lbrace S-P_0, S-P_0+R, S-P_0+2R, S-P_0+3R \rbrace$ must be equal, otherwise the maps $x\circ f^{-1}, y\circ f^{-1}$ would have more than four different poles in total. This means that every pole satisfies the equality $2P_0=S+kR$ for some $k\in \lbrace 0, 1, 2, 3\rbrace.$ Equivalently, every pole $P_0$ is a fixed point of some involution $i_k$ of the form $P\mapsto S-P+kR.$ Each involution $i_k$ has four fixed points on $E_m(\Qbar),$ because any two fixed points differ by an element from the $[2]$-torsion.

The involution $i_0,$ viewed on $\Dbar,$ maps an affine point $(x,y)=f^{-1}(P)$ to $(-x,y)=f^{-1}(S-P).$ It has two affine fixed points which have $x$-coordinate equal to zero on $\Dbar$, as well as two fixed points which are not affine on $\Dbar.$ Such points are either poles of $x\circ f^{-1}$ or poles of $y \circ f^{-1}.$ Using Magma\cite{bosma1997magma} we calculate the coordinates explicitly to obtain $R_1$ and $R_2.$ Computationally, we confirm $R_1$ and $R_2$ are poles of $x \circ f^{-1},$ that is, poles of $g.$

The involution $i_2,$ viewed on $\Dbar,$ maps an affine point $(x,y)=f^{-1}(P)$ to $(x,-y)=f^{-1}(S-P+2R).$ It has two affine fixed points which have $y$-coordinate equal to zero on $\Dbar,$ as well as two fixed points which are not affine on $\Dbar.$ These points must be poles of the map $y\circ f^{-1},$ that is, zeros of $g.$ Again, using Magma, we calculate the coordinates to obtain $S_1$ and $S_2.$

Since the poles of $x\circ f^{-1}$ are of order one, then the poles of $g$ are of order two. The same is true for poles of $y\circ f^{-1},$ that is, for zeros of $g.$ The last row of identities in the statement of the theorem is checked by Magma.
\end{proof}

\proposition \label{prop:4} There exists $h\in \QQ(E_m)$ such that $g\circ[2]=h^2.$
\proof Let $\tilde{h}\in \Qbar(E_m)$ such that 
\begin{align*}
\div \tilde{h}&=[2]^{\ast}((S_1)+(S_2)-(R_1)-(R_2))\\
&=\sum\limits_{T\in E_m[2]} (S'_1+T)+\sum\limits_{T\in E_m[2]} (S'_2+T)-\sum\limits_{T\in E_m[2]} (R'_1+T)-\sum\limits_{T\in E_m[2]} (R'_2+T),
\end{align*}
where $2S'_i=S_i,2R'_i=R_i$ and $[2]^{\ast}$ is the pullback of the doubling map on $E_m.$

Such $\tilde{h}$ exists because of Corollary 3.5 in Silverman\cite[III.3]{silverman2009arithmetic} stating that if $E$ is an elliptic curve and $D=\sum n_P(P)\in \text{Div}(E),$ then $D$ is principal if and only if
$$\sum_{P\in E} n_P=0 \text{ \hspace{5pt}and } \sum_{P\in E} [n_P]P=0,$$
where the second sum is addition on $E.$

The first sum being equal to zero is immediate, and for the second one we have
\[
\sum\limits_{T\in E_m[2]} (S'_1+T)+\sum\limits_{T\in E_m[2]} (S'_2+T)-\sum\limits_{T\in E_m[2]} (R'_1+T)-\sum\limits_{T\in E_m[2]} (R'_2+T)=
\]
\[
=[4](S'_1+S'_2-R'_1-R'_2)=[2](S_1+S_2-R_1-R_2)=
\]
\[
=[2](S_1-R_2+S_2-R_1)\stackrel{(\ast)}{=}[2](R+R)=\mathcal{O},
\]
where $(\ast)$ follows from the last row of identities in Proposition \ref{prop:3}.

Easy calculations give us $\div g\circ [2]=\div \tilde{h}^2$ which implies $C\tilde{h}^2=g\circ [2],$ for some $C\in \Qbar.$ Let $h:=\tilde{h}\sqrt{C}\in \Qbar(E_m)$ so that $h^2=g\circ [2].$ We will prove $h\in \QQ(E_m).$

First, we show that every $\sigma\in \Gal$ permutes zeros and poles of $\tilde{h}$. Let us check what $\sigma$ does to $S_1$ and $S_2.$ Since $S_1$ and $S_2$ are conjugates over $\QQ(\sqrt{q})$, the only possibilities for $S_1^\sigma$ are $S_1$ or $S_2.$ If $S_1^\sigma=S_1,$ then we must have $(S'_1)^\sigma=S'_1+T,$ where $T\in E_m[2],$ because $2((S'_1)^\sigma-S'_1)=(2S'_1)^\sigma-2S'_1=S_1^\sigma-S_1=\mathcal{O}.$ Thus $\sigma$ fixes $\sum\limits_{T\in E_m[2]} (S'_1+T)$. Since in this case we also know that $S_2^\sigma=S_2,$ we get that $\sigma$ fixes $\sum\limits_{T\in E_m[2]} (S'_2+T)$ as well.
 
 If  $S_1^\sigma=S_2$ it is easy to see that $$\left(\sum\limits_{T\in E_m[2]} (S'_1+T)\right)^\sigma=\sum\limits_{T\in E_m[2]} (S'_2+T) \text{ and } \left(\sum\limits_{T\in E_m[2]} (S'_2+T)\right)^\sigma=\sum\limits_{T\in E_m[2]} (S'_1+T).$$

Similar statements hold for $R_1$ and $R_2,$ so we conclude that $\tilde{h}$ is defined over $\QQ.$ Both $h$ and $\tilde{h}$ have the same divisor so $h$ is also defined over $\QQ$. Now we use the second statement from Theorem 7.8.3. in \cite{galbraith2012mathematics}:
\begin{theorem}
Let $C$ be a curve over a perfect field $k$ and let $f\in \overline{k}(C).$ 
\begin{enumerate}
\item If $\sigma(f)=f, ~~$ for each $\sigma \in \textrm{Gal}(\overline{k}/k)$ then $f\in k(C).$ 
\item If $\div(f)$ is defined over $k$ then $f=ch$ for some $c\in \overline{k}$ and $h\in k(C).$
\end{enumerate}
\end{theorem} 
From the second statement of the previous theorem we conclude that $h=c\cdot h'$ where $c\in \Qbar$ and $h'\in \QQ(E_m).$ We know that $c^2(h')^2=h^2=g\circ [2]$, and that $g\circ [2] (\mathcal{O})=(x_1^2-q)^2$ is a rational square. It follows that $\displaystyle c^2=\frac{(x_1^2-q)^2}{h'(\mathcal{O})^2}$ is a rational square as well, hence $c$ is rational. Finally, we have $h\in \QQ(E_m).$
\endproof
We end this section with a theorem which will handle rationality issues in Theorem \ref{thm:1}.
\begin{theorem} \label{thm:6} For all $P,Q\in E_m(\QQ)$ we have $g(P+Q)\equiv g(P)g(Q) \mod (\QQ^*)^2.$\\
In particular, if $P\equiv Q \mod 2E_m(\QQ) $ then $g(P)\equiv g(Q) \mod (\QQ^*)^2.$
\end{theorem}

\proof
Let $P',Q'\in E_m(\Qbar)$ such that $2P'=P$ and $2Q'=Q.$ We prove that
$$\frac{\sigma(h(P'+Q'))}{h(P'+Q')}=\frac{\sigma(h(P'))}{h(P')}\frac{\sigma(h(Q'))}{h(Q')}.$$

Following Silverman \cite[III.8]{silverman2009arithmetic}, assume $T\in E_m[2].$ From Proposition \ref{prop:4} it follows that $\displaystyle h^2(X+T)=g\circ [2] (X+T)=g\circ [2] (X)=h^2(X),$ for every $X\in E_m.$ This means that $\frac{h(X+T)}{h(X)}\in\{\pm 1\}.$ The morphism
$$ E_m \rightarrow \mathbb{P}^1, \qquad X\mapsto \frac{h(X+T)}{h(X)}$$
is not surjective, so by \cite[II.2.3]{silverman2009arithmetic} it must be constant.

For $ \sigma \in \Gal$ we have $\sigma(P')-P' \in E_m[2],\sigma(Q')-Q' \in E_m[2]$ and $\sigma(P'+Q')-(P'+Q') \in E_m[2].$ This holds since $2P'=P\in E_m(\QQ)$ and $2Q'=Q\in E_m(\QQ).$ Now we get 
$$\frac{\sigma(h(P'))}{h(P')}=\frac{h(\sigma(P'))}{h(P')}=\frac{h(P'+(\sigma(P')-P'))}{h(P')}=\frac{h(X+(\sigma(P')-P'))}{h(X)}.$$ Similarly $$\frac{\sigma(h(Q'))}{h(Q')}=\frac{h(X+(\sigma(Q')-Q'))}{h(X)},\quad \frac{\sigma(h(P'+Q'))}{h(P'+Q')}=\frac{h(X+(\sigma(P'+Q')-(P'+Q')))}{h(X)}.$$

Now
\begin{align*}
\frac{\sigma(h(P'+Q'))}{h(P'+Q')}&=\frac{h(X+(\sigma(P'+Q')-(P'+Q')))}{h(X)}\\
&=\frac{h(X+(\sigma(P'+Q')-(P'+Q')))}{h(X+\sigma(P')-P')}\frac{h(X+\sigma(P')-P')}{h(X)}\\
&=\frac{\sigma(h(Q'))}{h(Q')}\frac{\sigma(h(P'))}{h(P')}
\end{align*} 
by plugging in $X=P'+Q'-\sigma(P')$ for the first $X$ and $X=P'$ for the second one. This leads to $$\frac{h(P'+Q')}{h(P')h(Q')}=\frac{\sigma(h(P'+Q'))}{\sigma(h(Q'))\sigma(h(P'))}=\sigma\left(\frac{h(P'+Q')}{h(P')h(Q')}\right)$$ for every $\sigma\in \Gal.$ Now we conclude $$\frac{h(P'+Q')}{h(P')h(Q')}\in \QQ\implies h^2(P'+Q')\equiv h^2(P')h^2(Q') \mod (\QQ^*)^2.$$ Finally $$g(P+Q)=g\circ[2](P'+Q')= h^2(P'+Q')\equiv h^2(P')h^2(Q')=g(P)g(Q) \mod (\QQ^*)^2.$$

The second statement of the theorem follows easily from the first.\linebreak If $P=Q+2S_3,$ with $S_3\in E_m(\QQ),$ then $$g(P)=g(Q+2S_3)\equiv g(Q)g(S_3)^2\equiv g(Q) \mod (\QQ^*)^2. $$
\endproof

Theorem \ref{thm:6} was more difficult to prove compared to a similar statement in \cite[2.4.]{dujella2017more}. Their version of the function $g$ had a very simple factorization$\mod (\QQ^{\ast})^2,$ allowing them to use the $2-$descent homomorphism.

\section{Proofs of main theorems}\label{sec:3}

The main difficulty in the following proof is the issue of rationality of the quadruple. As we have mentioned, Theorem \ref{thm:6} will deal with this.

\emph{Proof of Theorem \ref{thm:1}:} From the assumptions on $(Q_1,Q_2,Q_3)$ we know \linebreak $(y_1^2-q)g(Q_1+Q_2+Q_3)$ is a square. We have
\begin{align*}
a^2&=\frac{g(Q_1)g(Q_2)g(Q_3)}{(x_1^2-q)^3m}=\frac{g(Q_1)g(Q_2)g(Q_3)(y_1^2-q)}{(x_1^2-q)^4(y_1^2-q)^2} \\
&\equiv g(Q_1+Q_2+Q_3)(y_1^2-q) \mod (\QQ^*)^2.
\end{align*}

The equivalence is a direct application of Theorem \ref{thm:6}.
This implies $a^2$ is a rational square so $a$ is rational, which in turn implies $b,c$ and $d$ are rational numbers, as noted in the introduction. Since $abcd=m\neq 0,$ none of the numbers $ a,b,c,d$ are zero, and the non-degeneracy criteria of $(Q_1,Q_2,Q_3)$ ensure that $a,b,c,d$ are pairwise different. Lastly, $ab+q=(x\circ f^{-1}(Q_1))^2$ (with similar equalities holding for other pairs of the quadruple). The previous statements prove the quadruple $(a,b,c,d)$ is a rational $D(q)$-quadruple. 

On the other hand, if $(a,b,c,d)$ is a rational $D(q)$-quadruple, then we can define the points $(Q_1,Q_2,Q_3)\in E_m(\QQ)^3$ in correspondence to $(a,b,c,d).$ Using the same identities$\mod (\QQ^*)^2$ as above, we get that $$(y_1^2-q)g(Q_1+Q_2+Q_3)\equiv a^2~~ \mod (\QQ^*)^2. $$ $\hfill \square$

To prove Theorem \ref{thm:2} we use the following lemma:
\begin{lemma}\label{lemma:7} Let $(a,b,c,d)$ be a rational $D(q)$-quadruple such that $abcd=m.$ There exists a point $(x_0,y_0)\in \mathcal{D}_m(\QQ),$ such that $x_0^2-q$ is a rational square.
\end{lemma}
\begin{proof}
From Theorem \ref{thm:1} we know that $(y_1^2-q)g(Q_1+Q_2+Q_3)$ is a square, where $(Q_1,Q_2,Q_3)\in E_m(\QQ)^3$ is the triple that corresponds to the quadruple $(a,b,c,d)$. 

Let $Q=Q_1+Q_2+Q_3.$ We have
\begin{align*}
(y_1^2-q)g(Q)&=(y_1^2-q)(x_1^2-q)((x\circ f^{-1}(Q))^2-q)=m\cdot (x\circ f^{-1}(Q))^2-q)\\
&=m\cdot \frac{m}{(y\circ f^{-1}(Q))^2-q}=m^2\frac{1}{(y\circ f^{-1}(Q))^2-q}.
\end{align*} 

Since the left hand side is a square, we conclude $(y\circ f^{-1}(Q))^2-q$ is a square as well. Now define $(x_0,y_0):=f^{-1}(Q+R).$ We know that 
\[
(y\circ f^{-1}(Q))^2-q\stackrel{(\ref{ness})}{=}(x\circ f^{-1}(Q+R))^2-q=x_0^2-q
\] so the claim follows.
\end{proof}

\emph{Proof of Theorem \ref{thm:2}:} 
Assume we have a rational $D(q)$-quadruple. Using Lemma \ref{lemma:7} there exists a point $(x_0,y_0)\in \mathcal{D}_m(\QQ)$ such that $x_0^2-q$ is a rational square. Since $x_0^2-q=k^2,$ then $q=x_0^2-k^2=(x_0-k)(x_0+k).$ Denote $u=x_0-k,$ then $x_0+k=q/u$ and by adding the previous two equalities together to eliminate $k,$ we get $x_0=\frac{q+u^2}{2u}.$ Denoting $t=y_0$ we get $$m=(x_0^2-q)(y_0^2-q)=\left(\left(\frac{q+u^2}{2u}\right)^2-q\right)(t^2-q)=\left(\frac{q-u^2}{2u}\right)^2(t^2-q).$$

Now, let $m=\left(\frac{q-u^2}{2u}\right)^2(t^2-q)$ for some rational $(t,u).$ Denote $y_1=t, x_1=\frac{q+u^2}{2u}.$ It is easy to check that $(x_1^2-q)(y_1^2-q)=m,$ so there is a rational point $(x_1,y_1)\in\mathcal{D}_m(\QQ)$ such that $x_1^2-q=\left(\frac{u^2-q}{2u}\right)^2$ is a square. We use this point $(x_1,y_1)=\colon P_1$ to define the map $f:\mathcal{D}_m\to E_m.$ Let $Q_1=R+S, Q_2=2S$ and $Q_3=3S.$ The sets $G\cdot Q_i$ are disjoint and $g(Q_1+Q_2+Q_3)(y_1^2-q)=g(R+6S)(y_1^2-q)\equiv g(R)(y_1^2-q)=((x_1^2-q)(y_1^2-q))(y_1^2-q)$ mod $(\QQ^*)^2$ is a rational square. The points $(Q_1,Q_2,Q_3)$ satisfy the conditions of Theorem \ref{thm:1} giving us a rational $D(q)$ quadruple.$\hfill \square$

\begin{remark}\label{rem:8} The condition $m=\left(\frac{q-u^2}{2u}\right)^2(t^2-q)$ is equivalent to the fact that there exists $(x_0,y_0)\in\mathcal{D}_m(\QQ)$ such that $x_0^2-q$ is a square. This was proven in the preceding theorems.
\end{remark}

\section{Examples}\label{sec:4}

There are plenty of examples where $m=(x_1^2-q)(y_1^2-q)$ for some rational $x_1$ and $y_1,$ such that there does not exist a rational $D(q)$-quadruple with product $m.$ Equivalently, $m$ cannot be written as $(x_0^2-q)(y_0^2-q)$ such that $x_0^2-q$ is a square.

According to Theorem \ref{thm:1}, to find out whether there is a rational $D(q)$-quadruple with product $m,$ one needs to check whether there is a point $T'\in E_m(\QQ)$ such that $g(T')(y_1^2-q)$ is a square. Theorem \ref{thm:6} tells us that we only need to check the points $T \in E_m(\QQ)/2E_m(\QQ),$ which is a finite set. If for some explicit $q,m$ we know the generators of the group $E_m(\QQ) / 2E_m(\QQ),$ we can determine whether there exist rational $D(q)$-quadruples with product $m,$ and parametrize them using points on $E_m(\QQ).$ For such computations we used Magma. 

Let $q=3, x_1=5$ and $y_1=7$ making $m=(5^2-3)(7^2-3)=1012.$ The rank of $E_m$ is two, $E_m$ has one torsion point of order four, giving us in eight points in total to check. None of the points $T \in E_m(\QQ)/2E_m(\QQ)$ satisfy that $g(T)(y_1^2-q)$ is a square, so there are no $D(3)$-quadruples with product $1012.$

On the other hand, take $q=-3,x_1=1$ so that $x_1^2-q=4$ and let $y_1=t$ which makes $m=4\cdot(t^2+3).$ The point $S$ is a point of infinite order on $E_m(\QQ(t)),$ and the triple $(Q_1,Q_2,Q_3)=(S+R,2S,3S)$ satisfies the conditions of Theorem \ref{thm:1}. We obtain the following family:\\
\begin{align*}
a&=\frac{2\cdot (3 + 6 t^2 + 7 t^4)\cdot (27+162t^2+801t^4+1548t^6+1069t^8+306t^{10}+183t^{12})}{(3 + t^2)\cdot (1 + 3 t^2)\cdot (9+9t^2+19t^4+27t^6)\cdot(3+27t^2+33t^4+t^6)}, \\
b&=\frac{(3 + t^2)^2\cdot (1 + 3 t^2)\cdot (9+9t^2+19t^4+27t^6)\cdot(3+27t^2+33t^4+t^6)}{2\cdot (3 + 6 t^2 + 7 t^4)\cdot (27+162t^2+801t^4+1548t^6+1069t^8+306t^{10}+183t^{12})},\\
c&=\frac{2\cdot (3 + 6 t^2 + 7 t^4)\cdot(3 + 27 t^2 + 33 t^4 + t^6)\cdot (9 + 9 t^2 + 19 t^4 + 27 t^6)}{(3 + t^2)\cdot (1 + 3 t^2)\cdot(27 + 162 t^2 + 801 t^4 + 1548 t^6 + 1069 t^8 + 
   306 t^{10} + 183 t^{12})},\\
d&=\frac{2\cdot(3 + t^2)\cdot (1 + 3 t^2)\cdot(27 + 162 t^2 + 801 t^4 + 1548 t^6 + 1069 t^8 + 
   306 t^{10} + 183 t^{12})}{(3 + 6 t^2 + 7 t^4)\cdot(3 + 27 t^2 + 33 t^4 + t^6)\cdot (9 + 9 t^2 + 19 t^4 + 27 t^6)}.
\end{align*}
We can generalize the example above by setting $q=q, y_1=t, x_1=\frac{q+u^2}{2u}.$ The triple of points $(S+R,2S,3S)$ satisfies the conditions of Theorem \ref{thm:1} and we can calculate an explicit family of rational $D(q)$-quadruples with product $m,$ but this example is too large to print (the numerator of $a$ is a polynomial in the variables $(q,t,u)$ of degree forty).

All the computations in this paper were done in Magma \cite{bosma1997magma}.
\section{Acknowledgements}

The authors were supported by the Croatian Science Foundation under the project no.~1313. \\
The second author was supported by the QuantiXLie Centre of Excellence, a project cofinanced by the Croatian Government and the European Union through the European Regional Development Fund - the Competitiveness and Cohesion Operational Programme (Grant KK.01.1.1.01.0004).

\bibliography{Rational_D_q_quadruples_indagationes}

\end{document}